\documentclass[12pt]{amsart}

\usepackage{amsfonts,amssymb, latexsym, epsfig, epic}
\usepackage{amsmath,amsthm,lipsum,commath,graphicx,setspace}

\newtheorem{thm}{Theorem}
\theoremstyle{plain}

\newtheorem{lemma}{Lemma}

\newtheorem{remark}{Remark}

\begin{document}

\title{Supercoiled Tangles and Stick Numbers of 2-Bridge Links}
\author{Erik Insko}
\address{Dept. of Mathematics\\
Florida Gulf Coast University\\ Fort Myers, FL 33965}

\email{einsko@fgcu.edu}
\author{Rolland Trapp}
\address{Dept. of Mathematics\\
 5500 University Pkwy\\
San Bernardino, CA 92407}
\email{rtrapp@csusb.edu}

\begin{abstract}
Utilizing both twisting and writhing, we construct integral tangles with few sticks,  leading to an efficient method
for constructing polygonal 2-bridge links.  Let $L$ be a two bridge link with crossing number $c$, stick number $s$, and $n$ tangles. 
It is shown that $s \le \frac{2}{3}c + 2n+3$.  We also show that if $c > 12n+3$, then minimal stick representatives do not admit minimal crossing projections. 

\vspace{12pt}


\end{abstract}

\maketitle
\section{Introduction}

The stick number $s(L)$ of a link $L$ is the fewest number of line segments needed to represent $L$
in piecewise-linear fashion.  The crossing number $c(L)$ is the fewest number of crossings in any diagram
of $L$.  The exact stick and crossing numbers of links are difficult to ascertain in general.  The advent of knot polynomials, however, provided an effective tool for determining the crossing number of large classes of
links.  Knot polynomials were used to find the crossing number of alternating and adequate links 
\cite{ka, mu2,th,lt}, as well as that of torus links (see \cite{mu}).

Similarly, finding the precise stick number of classes of links is a difficult task.  This was accomplished for the family of  
torus links $L_{p,q}$, with $2 \le p \le q \le 2p$ by Jin in 1997 \cite{ji}
and extended to $2p < q \le 3p$ in 2013 by Johnson, Mills, and Trapp \cite{jmt}.  The basic idea was to construct the link using few sticks, then show the construction to be optimal
using the superbridge number.  The stick numbers of $L_{p,p+1}$ torus knots, together with their compositions, were determined by Adams, Brennan, Greilsheimer, and Woo in 1997 \cite{ab},
and those of $L_{p,2p+1}$ by Adams and Shayler in 2009 \cite{as}.  These results are obtained by again constructing the link efficiently, then using projections and total curvature 
(or bridge number) considerations to prove them optimal.  In both cases, the lower bounds come from geometric properties of torus links.
Thus the classes of links with known stick number are rare and enjoy specific geometric properties.

In addition to stick numbers of torus links, upper bounds for the stick numbers of 2-bridge links have been the object of some study.
In 1998, Furstenberg, Lie, and Schneider found that $s(L) \le c(L) +2$
for 2-bridge links with few tangles \cite{fls}. In the same year, McCabe found that $s(L) \le c(L)+3$ for arbitrary 2-bridge links, a considerable extension of 
results of Furstenberg et al. resulting from the addition of one stick \cite{mc}.  McCabe then conjectured
that her upper bound could be improved to $c(L) + 2$ and noted that her construction leads to minimal crossing projections.
McCabe's conjecture was proven by Huh, No, and Oh in 2011, and their construction again produces polygonal representatives that
admit minimal crossing projections \cite{hno}.  Both McCabe and Huh et al. construct $c$-crossing polygonal tangles
using $c+1$ sticks as in Figure \ref{oldtangle}. This approach arises naturally out of minimal crossing projections, and is reminiscent of DNA that is in $\lq\lq$linear form" as described by Brauer, Crick, and White \cite[p. 124]{bcw}. The picture of supercoiled DNA in Figure~\ref{dna} suggests an alternate construction.

\begin{figure}[h]
 \centering
 \includegraphics[height=1in]{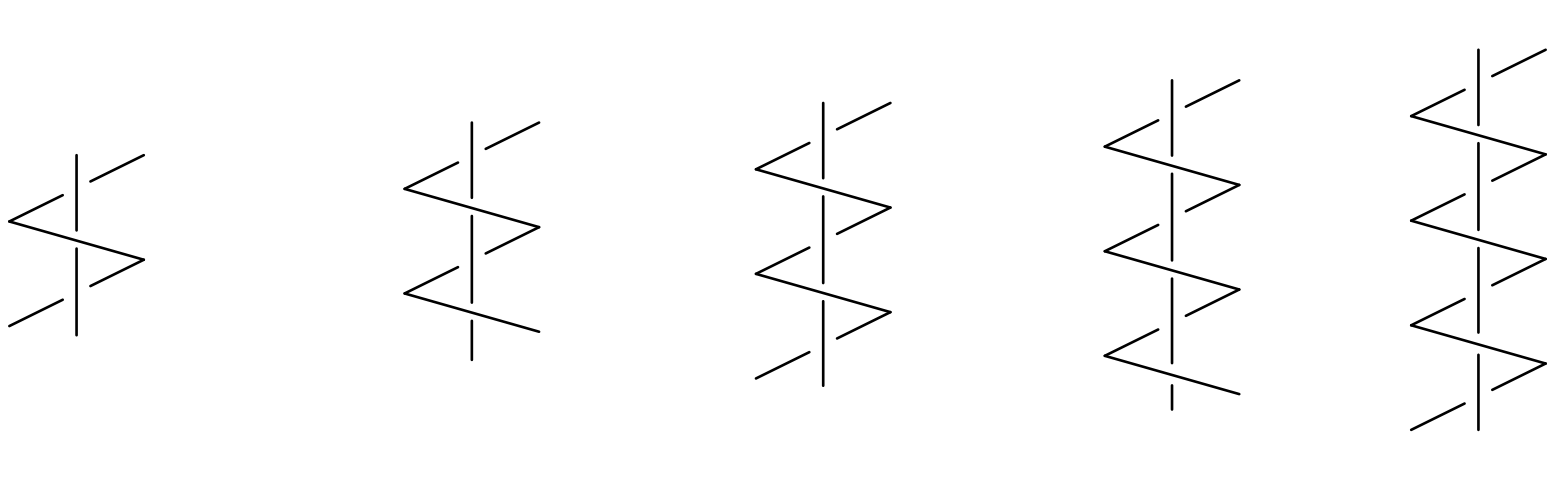}
 \caption{Existing construction of integral tangles} \label{oldtangle}
\end{figure}

Motivated by supercoiled DNA and the work of Johnson et al. for constructing polygonal cable links \cite{jmt}, we construct supercoiled integral tangles.  
In the supercoiled situation both twisting and writhing contribute to the crossing number of a tangle, in contrast with Figure~\ref{oldtangle} where all 
crossings result from twisting.  The result is a construction that uses roughly $\frac{2}{3}c$ sticks to build a $c$-crossing integral tangle, see Figure~\ref{create}.

\begin{figure}[h]
 \centering
 \includegraphics[height=1.5in]{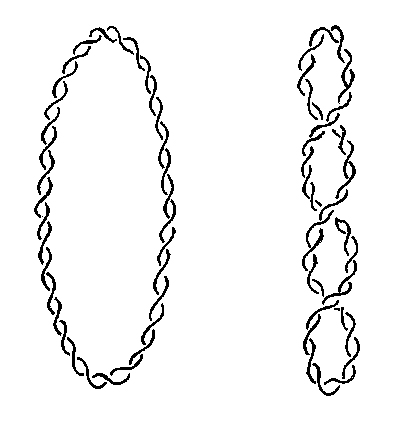}  \includegraphics[height=1.5in]{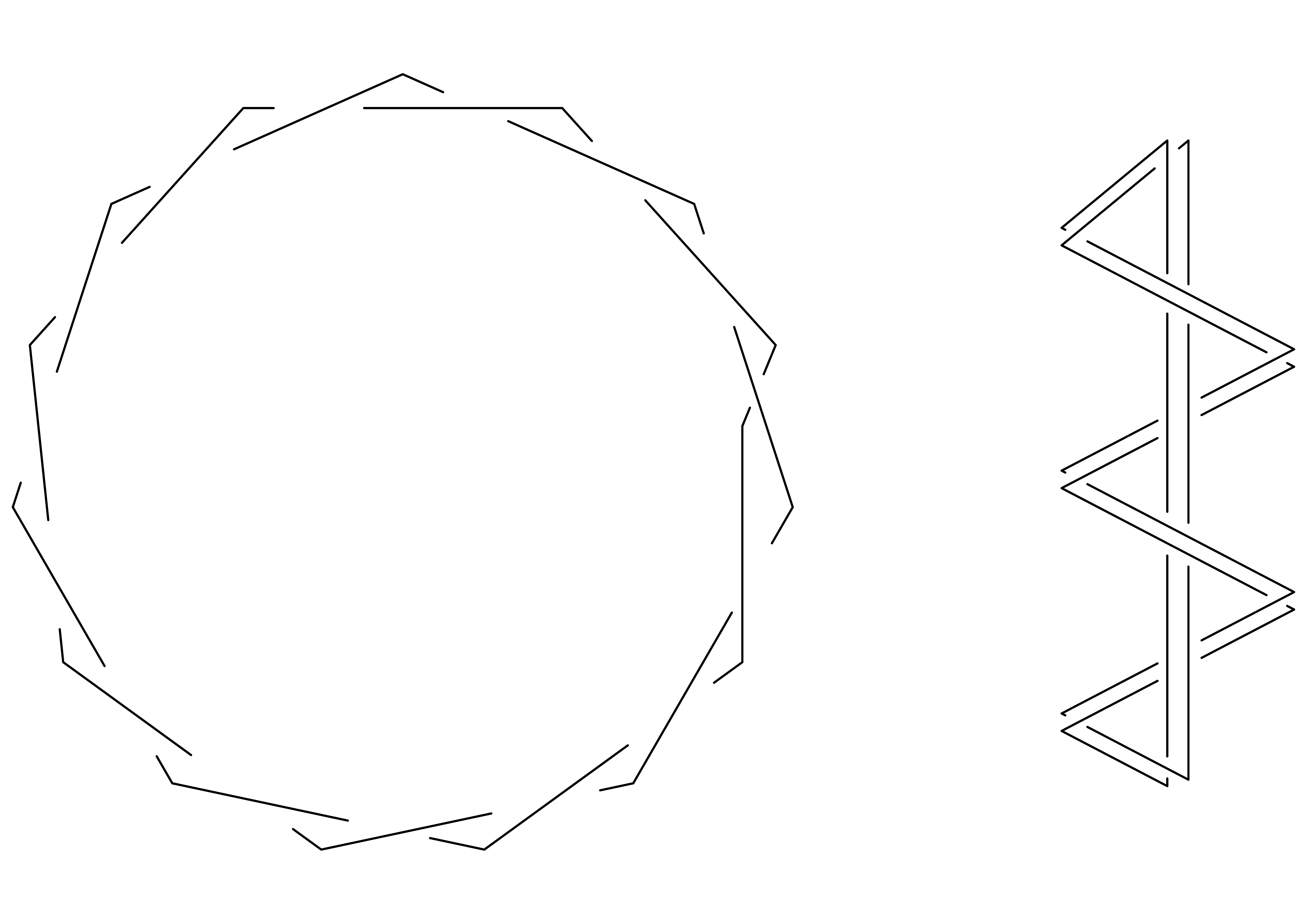} 
 % Supercoil5.jpg: 910x320 pixel, 150dpi, 15.41x5.42 cm, bb=0 0 437 154
 \caption{Picture of supercoiling in DNA double helix and in the stick knot $T_{2,15}$, \cite{wi}.} \label{dna}
\end{figure}

These supercoiled tangles can be clipped together to construct polygonal rational links, which leads to improved upper bounds for stick numbers of 2-bridge  links in Section \ref{2bridgelinks}.  
The main result of the paper, Theorem \ref{thm:2BridgeUpperBound}, essentially shows that $s(L) \le \frac{2}{3}c(L) + 2n + 3$ for a 2-bridge link $L$ with $n$ tangles.
An application of the improved bound is given in Section \ref{crossings}.  Previous polygonal constructions of 2-bridge links admitted minimal crossing projections.  In contrast, we use
Theorem \ref{thm:2BridgeUpperBound} to
show that minimal stick representatives of 2-bridge links do not admit minimal crossing projections when
the crossing number is large relative to the number of tangles.

\section{Upper Bounds for 2-Bridge Links} \label{2bridgelinks}

In this section we describe a piecewise linear construction of supercoiled integral tangles and how to glue them together to build polygonal 2-bridge links.  
The result is an improved upper bound for the stick
number of 2-bridge links with crossing number roughly six times the number of tangles or more.  See 
Theorem \ref{thm:2BridgeUpperBound} for a precise statement.

In what follows we describe how to construct a supercoiled integral tangle with negative crossings. 
To construct one with positive crossings, simply take the mirror image of the ones constructed here.  
Start with a piecewise linear \emph{core} and label the $m+2$ edges 
$e_1, \ldots, e_{m+2}$.  
Figure~\ref{core} shows such a core viewed from two different directions.  We choose to index the
edges with $m+2$ since the first and last edges of the tangle have a free endpoint
(see Remark~\ref{remark1} after Lemma~\ref{lemma:IntegralTangle} as well).
Also, let $v_i$ denote the vertex at the intersection of $e_i$ and $e_{i+1}$.  
We will add two new vertices $o_i$ and $u_i$ next to each vertex $v_i$, and then we connect the $o_i$'s and $u_{i}$'s to create the polygonal strands of the tangle. 

\begin{figure}[h]
 \centering
 \includegraphics[height=2in]{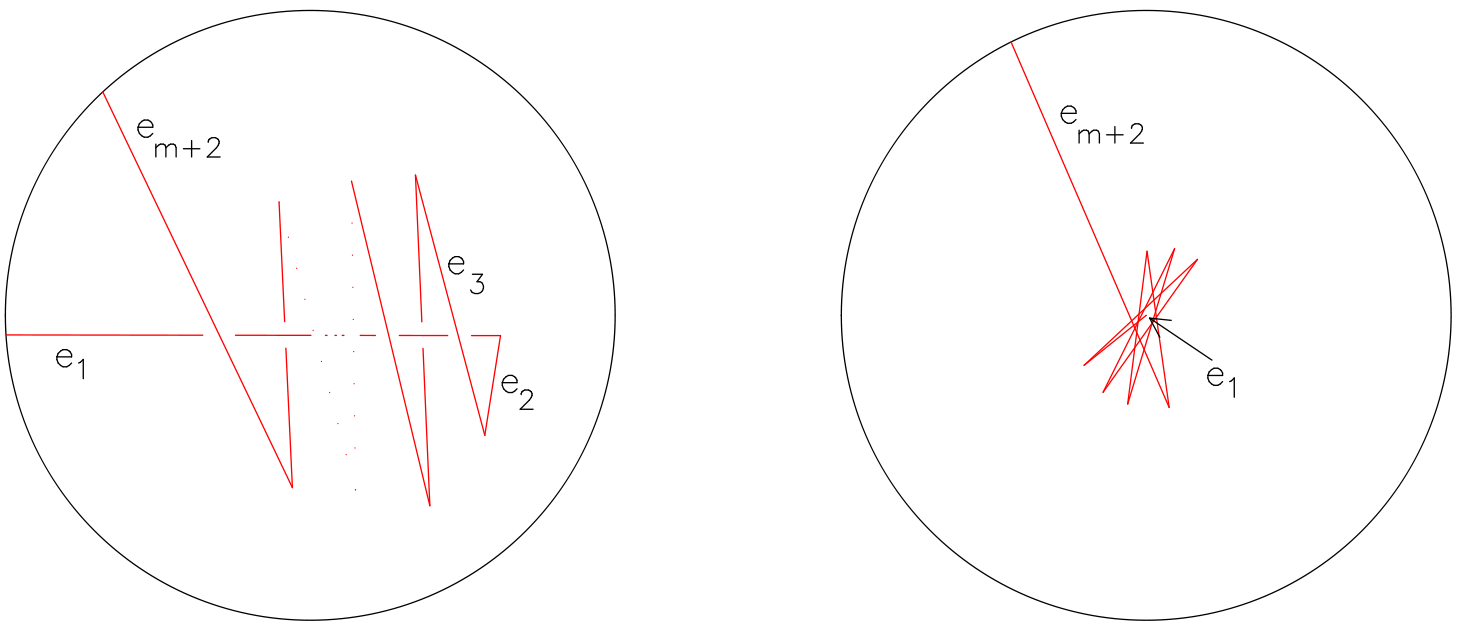}
 % Core.pdf: 3370x2383 pixel, 72dpi, 118.89x84.07 cm, bb=0 0 3370 2383
 \caption{Core of an integral tangle with writhe $w=- m$} \label{core}
\end{figure}

The technique for placing the new vertices is now described, and is analogous to the construction
of polygonal 2-cables used by Johnson et al. \cite{jmt}.  Figure~\ref{PQ} summarizes the description.
The edges $e_i$ and $e_{i+1}$ of the core determine a plane $P$. 
Let $\ell$ be the line bisecting the acute angle 
defined by $e_i$ and $e_{i+1}$, and $k$ be the line perpendicular to $\ell$ at $v_i$ in $P$.  Let $o_i$ and $u_i$ be the two points on the line $k$ 
that are distance $\epsilon$ away from the vertex $v_i$.  Rotate the line $k$ slightly about the line $\ell$ so that it lies in a plane $Q$ with $o_i$ over, and $u_i$ under, $P$ in the projection 
direction $z$ (see Figure~\ref{PQ}).  This determines the placement of the new vertices.

\begin{figure}[h]
 \centering
 \includegraphics[height=2 in]{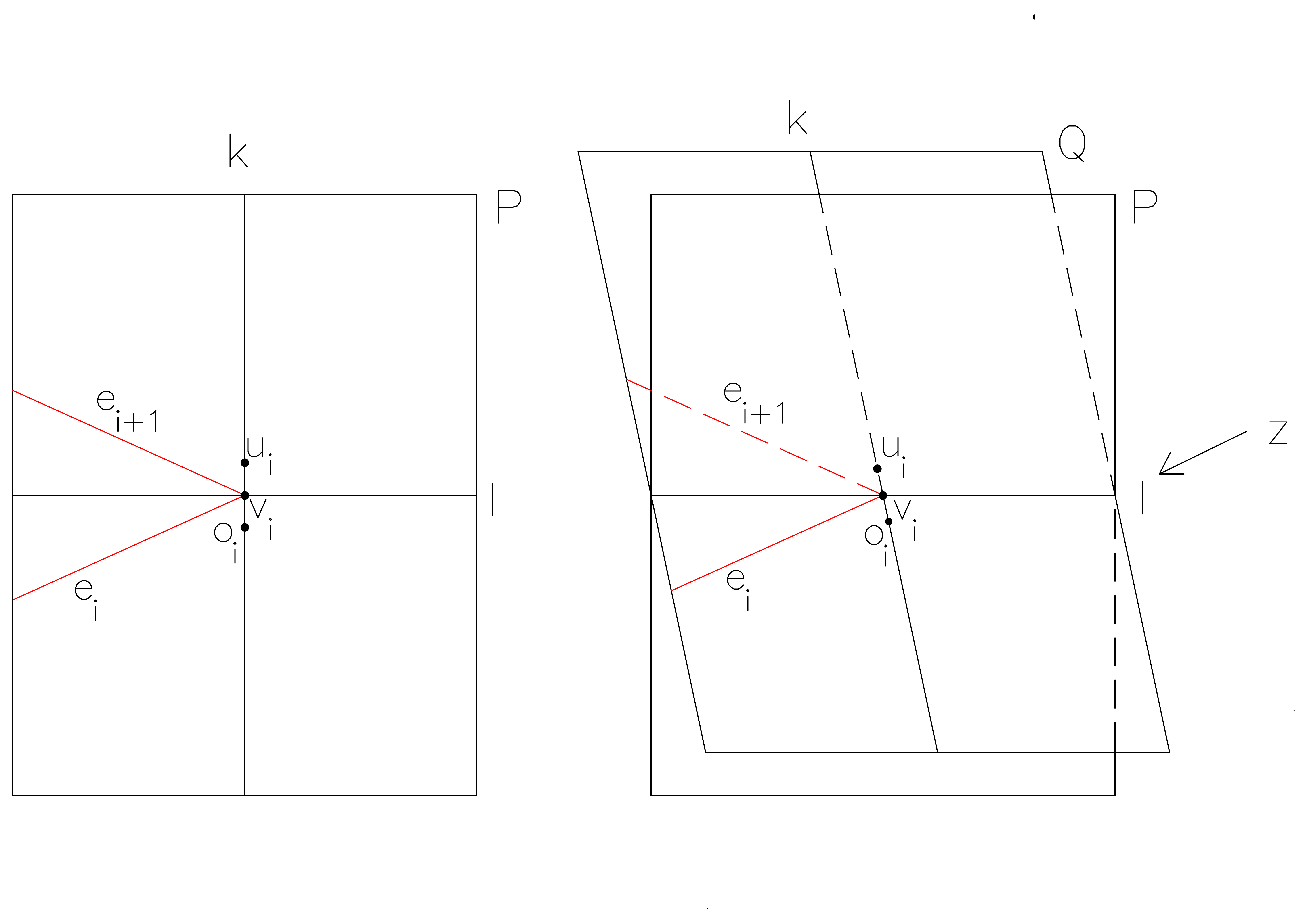}
 % Turn.pdf: 3370x2383 pixel, 72dpi, 118.89x84.07 cm, bb=0 0 3370 2383
 \caption{Adding the vertices $o_i$ and $u_i$} \label{PQ}
\end{figure}

Now that the vertices have been placed, create the edges of the tangle by connecting $o_i$ to $u_{i+1}$ and $u_i$ to $o_{i+1}$.  We remark that the distance $\epsilon$ from $v_i$ to the new vertices should be
chosen small enough that the entire polygonal tangle lies inside an embedded tube around the core.  This
choice of $\epsilon$, together with the relative positions of planes $P$ and $Q$, imply that the tangle 
appears as in the first diagram of Figure~\ref{create}.  As in that figure, 
this piecewise linear tangle can be isotoped to a smooth integral tangle.

\begin{figure}[h]
 \centering
 \includegraphics[height=1.5 in]{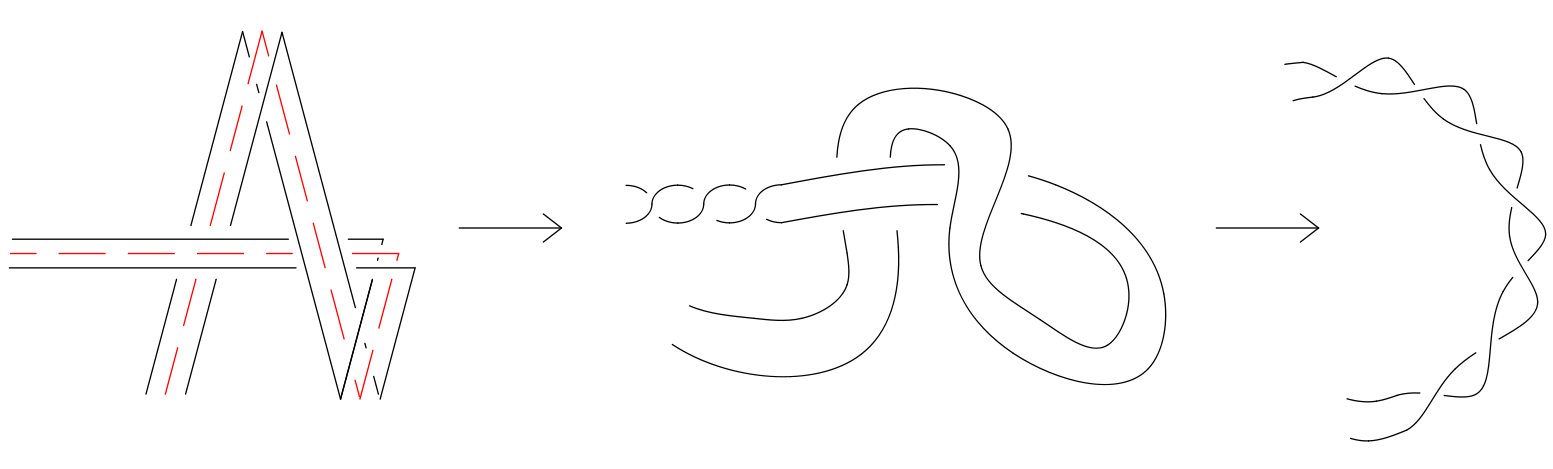}
 % Turn.pdf: 3370x2383 pixel, 72dpi, 118.89x84.07 cm, bb=0 0 3370 2383
 \caption{Isotopy to smooth tangle} \label{create}
\end{figure}

\begin{lemma}\label{lemma:IntegralTangle}
An integral tangle with $c$ crossings can be made using $s$ sticks, where 
\[
s =  \begin{cases}
\frac{2}{3}c+4& c\equiv 0\pmod 3\\ 
\frac{2}{3}c+\frac{10}{3}& c\equiv 1\pmod 3\\
\frac{2}{3}c+\frac{11}{3}& c\equiv 2\pmod 3
\end{cases}
\]
\end{lemma}

\begin{proof}
We will treat the case where $c \equiv 1 \pmod 3$ first.  (This is the case where $o_i$ is connected to $u_{i+1}$ and $u_i$ is connected to $o_{i+1}$ for all $i$.)  The behavior of the new edges
near the vertices and self-crossings of the core allow us to count the number of crossings in
the constructed integral tangle.
The core has $m+2$ edges yielding $m+1$ internal vertices $v_1, \ldots, v_{m+1}$ and $m$ self-crossings.  
The half-twist at
each internal vertex $v_i$ of the core contributes one crossing to the integral tangle.  Further, each self-crossing of the core yields two sets of parallel strands in the tangle that cross each other.
Such a conformation contributes a full twist, or two crossings, to the integral tangle (as in Figure~\ref{create}).  
Thus in the tangle there are a total of $c = (m+1) + 2m =  3m+1$ crossings. Since each of the $m+2$ edges of the core corresponds to $2$ sticks of the integral tangle, the stick number of the tangle is $s = 2m+4$.
These equations show that $s = \frac{2}{3} c+ \frac{10}{3}$ in the case $c \equiv 1 \pmod 3$.

\begin{figure}[h]
 \centering
 \includegraphics[width=5 in]{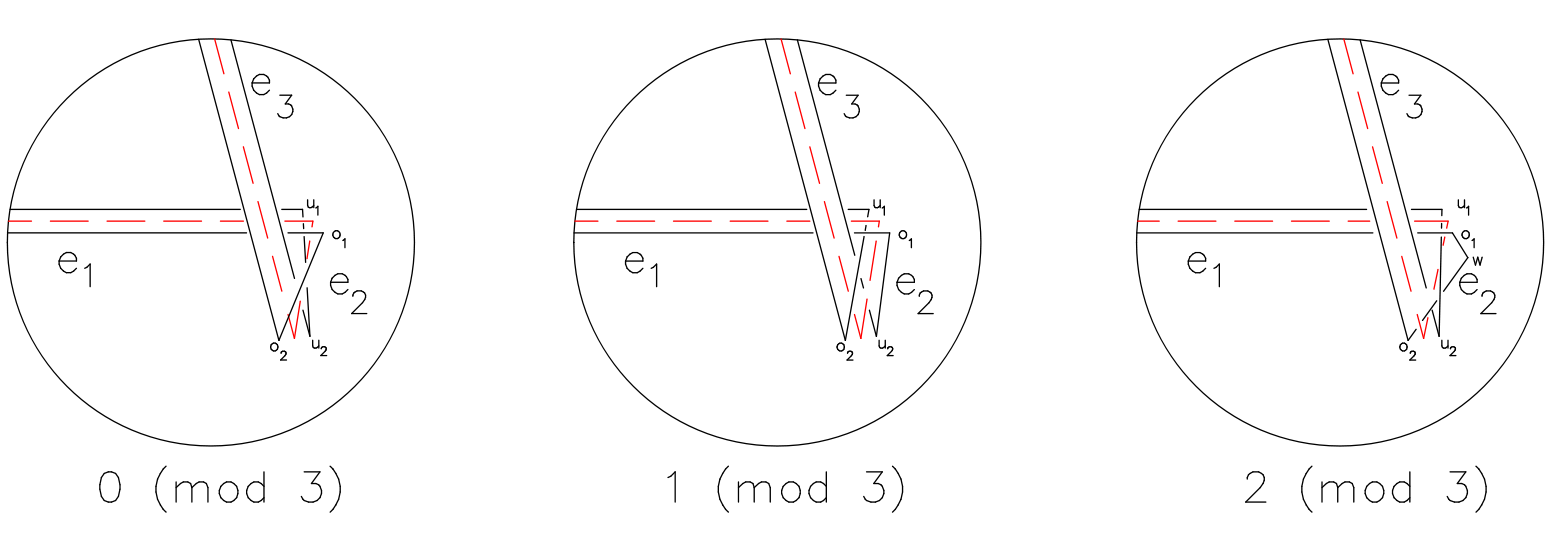}
 % Turn.pdf: 3370x2383 pixel, 72dpi, 118.89x84.07 cm, bb=0 0 3370 2383
 \caption{Base Constructions} \label{core1}
\end{figure}

To construct a tangle with crossing number $c \equiv 0 \pmod 3$, we connect $o_1$ to $o_2$ and $u_1$ to $u_2$ as in Figure~\ref{core1}, and leave the remaining connections the same.  
This uses the same number of sticks as the previous case, but reduces
the crossing number of the tangle by one.  Reasoning as above gives the desired result, $s = \frac{2}{3}c + 4$. 

To construct a tangle with crossing number $c \equiv 2 \pmod 3$, we add a new vertex $w$ that lies below $u_2$ and $o_1$.  
The vertex $w$ should be chosen so that the edge from it to 
$o_2$ passes between the two edges adjacent to $u_2$ as in the third image of 
Figure~\ref{core1}.
Connect $o_1$ to $w$, $w$ to $o_2$, and $u_1$ to $u_2$, and leave the remaining connections the same.
This uses one more stick than in the $c \equiv 0 \pmod 3$ case, but
yields $3m+2$ crossings. Thus $s = \frac{2}{3} c+ \frac{11}{3}$ in this case. 
\end{proof}

\begin{remark} \label{remark1}
Note from the construction in the proof that a tangle with $c$ crossings comes from a polygonal core with $m+2$ edges where $m = \left \lfloor \frac{c}{3} \right \rfloor$.
\end{remark} 

Using the construction of supercoiled integral tangles in Lemma \ref{lemma:IntegralTangle}, we will describe a method of efficiently constructing polygonal $2$-bridge links.   Any 2-bridge knot or link has a reduced
alternating diagram as in Figure~\ref{conway}, which we call a \emph{standard diagram}. Further, the number of tangles $n$ can be chosen to be an odd integer \cite{bz}.
Since the projection is reduced and alternating, it has minimal crossing number \cite{ka, mu2, th}.

\begin{figure}[h]
 \centering
 \includegraphics[height=2in]{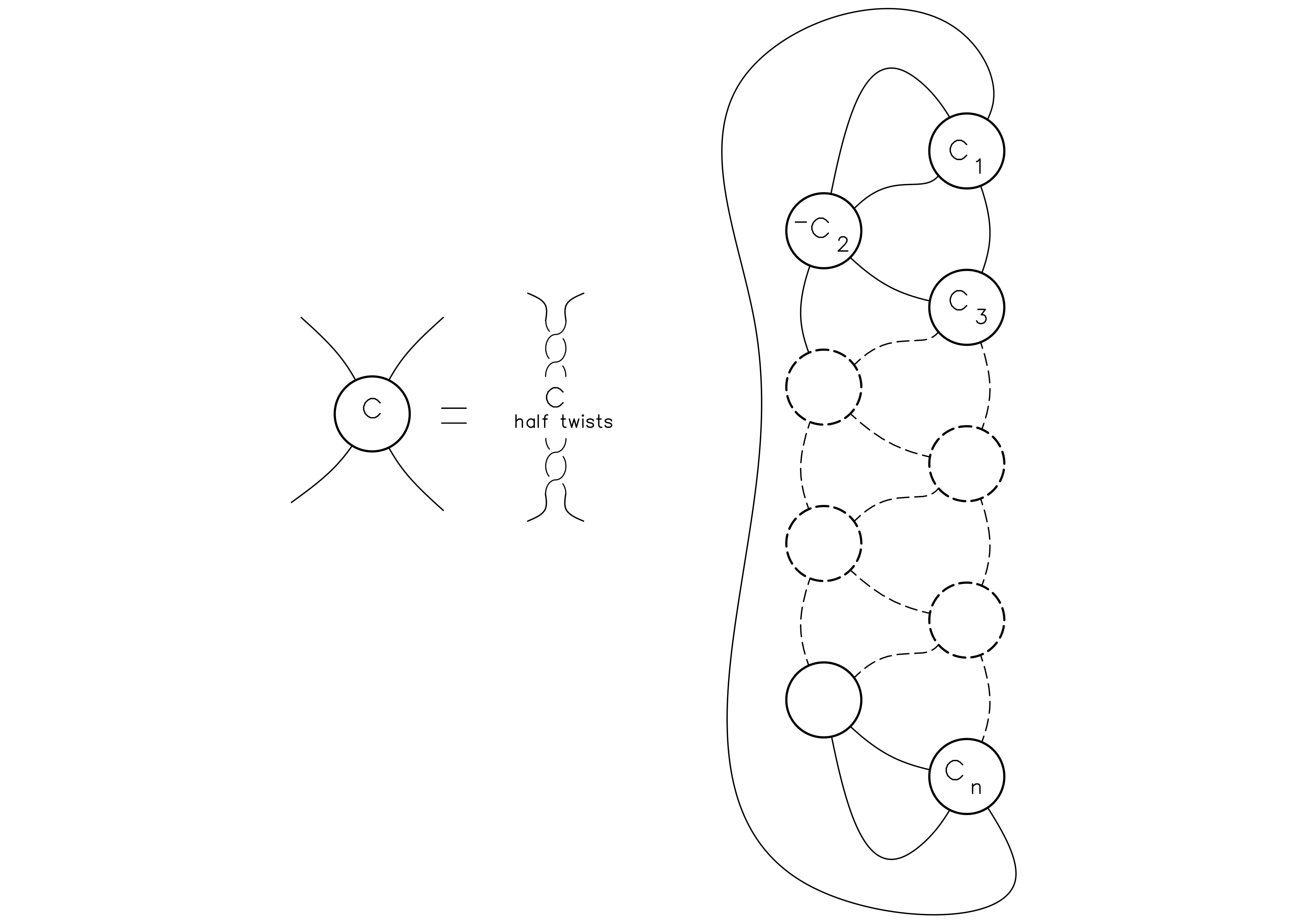}
 % Bubbles3.pdf: 3370x2383 pixel, 72dpi, 118.89x84.07 cm, bb=0 0 3370 2383
 \caption{Standard diagram of a 2-bridge link} \label{conway}
\end{figure}

We now construct polygonal representatives of 2-bridge links.  
To do so we will first turn a standard diagram into a linear planar graph with $2n+3$ edges, and then
insert supercoiled tangles to result in the desired 2-bridge link. 

\begin{figure}[h]
 \centering
 \includegraphics[height=3 in]{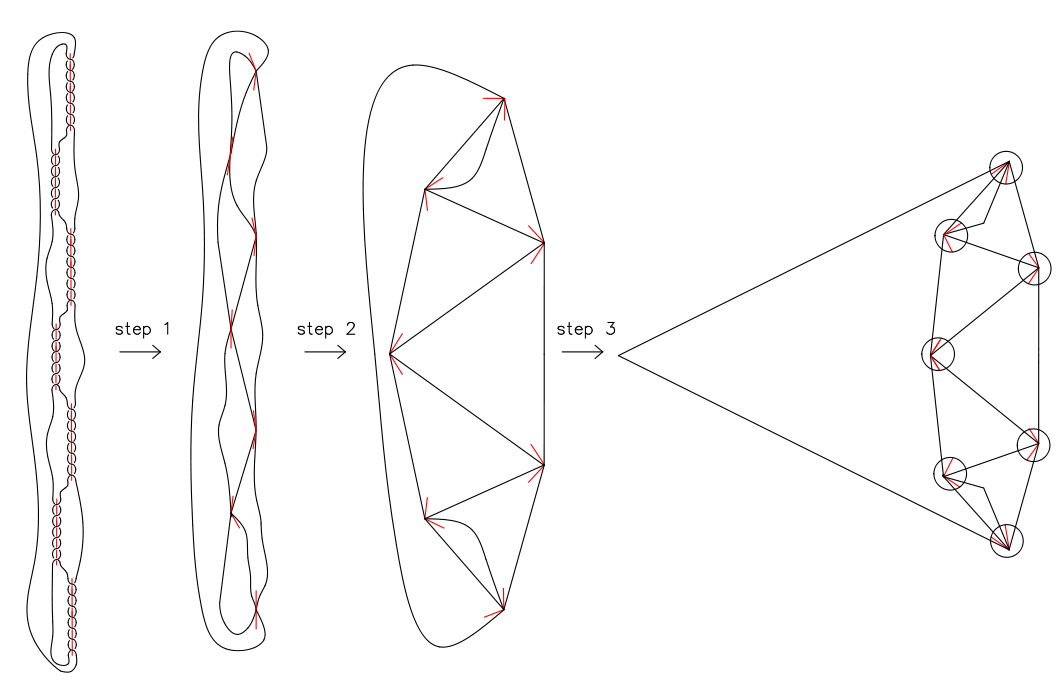}
 % SmoothToLinearStep2.pdf: 3370x2383 pixel, 72dpi, 118.89x84.07 cm, bb=0 0 3370 2383
 \caption{Polygonal embedding of a standard 2-bridge diagram} \label{figure:algorithm}
\end{figure}

Start with the standard diagram of a 2-bridge link in Figure~\ref{conway}, and shrink each integral tangle 
to a vertex.  The regions around each vertex are decorated with a core to indicate which direction the 
integral tangle should be inserted.
The result is a 4-valent planar graph decorated with a core extending out of each vertex in two of its four surrounding regions (see Step 1 of Figure~\ref{figure:algorithm}).  

Isotope the vertices corresponding to negative tangles to the left and the ones corresponding to positive tangles to the right, so that the vertices all lie on a convex curve in the plane.
As in Step 2 of Figure~\ref{figure:algorithm}, all but three of the edges can be isotoped to be line segments. The remaining three edges can be represented by two line segments each as in the final diagram of
 Figure~\ref{figure:algorithm}.
The result is a polygonal embedding of the $4$-valent graph with $2n+3$ sticks. 

Now that we have a decorated, planar $4$-valent graph, we set about the task of inserting the supercoiled tangles.  
We will describe the insertion for a positive integral tangle, and the mirror image gives the corresponding construction for negative integral tangles.

Place a 3-ball around each vertex. Inside each 3-ball the graph looks like the first image in Figure~\ref{break}.
Replace the single vertex with two vertices stacked directly on top of each other,
 so that the two northern edges and core stay attached to the bottom vertex and the two southern strands and core stay attached to the top vertex. 
We then shift the vertices horizontally with respect to the projection direction so that cores and edges cross each other as in the second step of Figure~\ref{break}.

\begin{figure}[h]
 \centering
 \includegraphics[width=4.5in]{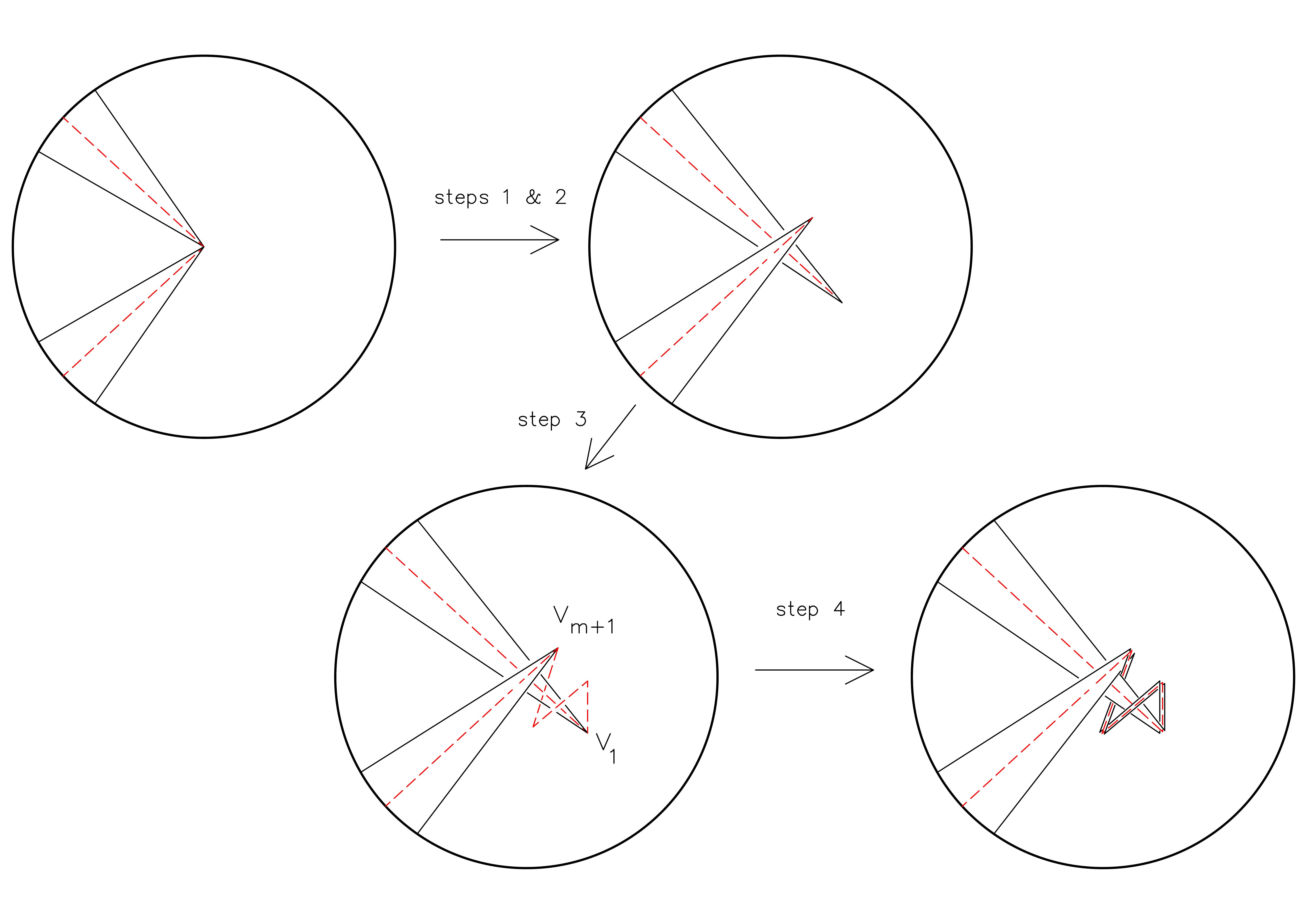}
 % Step0.pdf: 3370x2383 pixel, 72dpi, 118.89x84.07 cm, bb=0 0 3370 2383
\caption{Inserting integral tangles}  \label{break}
\end{figure}

Perform the previous two steps on all vertices concurrently.    These steps can be performed in such a way that they do not alter the topology of the graph outside the 3-balls.  

Recall from Remark \ref{remark1} that constructing a tangle with $c$ crossings requires a core with $m+2$ edges where $m = \left \lfloor \frac{c}{3} \right \rfloor$.
We already have two edges of the core inside the 3-ball, and must add $m$ edges to complete the core. 

Label the bottom vertex $v_1$ and the top vertex $v_{m+1}$.
Add a polygonal core that starts at the vertex $v_1$, writhes around the strands and the core that are adjacent to the vertex $v_1$
a total of $m-1$ times and connects to the other vertex $v_{m+1}$ as in the third step of Figure~\ref{break}.
This core will act as the core of the supercoiled tangle in the construction decribed prior to Lemma~\ref{lemma:IntegralTangle}.
For $1 < i < m+1$ place the $o_i$ and $u_i$ as in that construction.  

The placement of the vertices $o_1, u_1, o_{m+1}, u_{m+1}$ requires some extra care.
Attach edges from $o_2$ and $u_2$ to the vertex $v_1$, creating a double point in the tangle. 
One of the two possible $\epsilon$-resolutions of this double point results in desired crossings 
in the diagram.  Label the appropriate points of that resolution $o_1$ and $u_1$ as in Figure~\ref{3d}.
Use the same technique for placing $o_{m+1}$ and $u_{m+1}$.  We may choose $\epsilon$ small enough so that the topology outside the $3$-balls is unchanged.
When $c \equiv 0$ or $1 \pmod 3$ we are done.  If $c \equiv 2 \pmod 3$ an extra stick must be added as in the proof of Lemma \ref{lemma:IntegralTangle}.
Applying this construction inside each 3-ball results in a polygonal link that is isotopic to our original $2$-bridge link.

\begin{figure}
 \centering
 \includegraphics[width=4in]{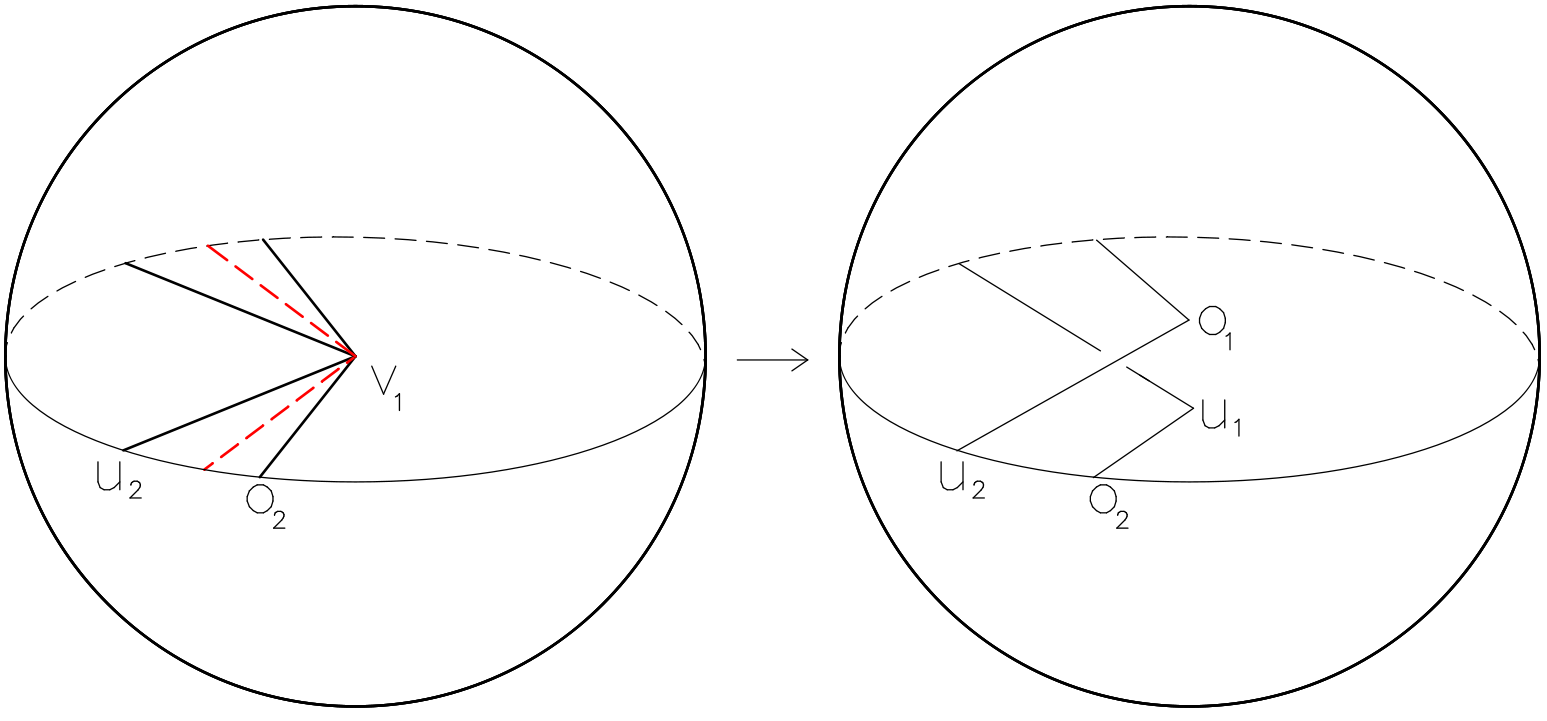}
 % resolution.png: 1291x624 pixel, 72dpi, 45.54x22.01 cm, bb=0 0 1291 624
 \caption{$\epsilon$-Resolution of a double point in tangle} \label{3d}
\end{figure}

\begin{thm}\label{thm:2BridgeUpperBound}
Let $L$ be a rational link given by the integers $c_1, c_2, \dots,c_n$, and suppose $P$ of the integers satisfy $c_i \equiv 0 \pmod 3$, $Q$ satisfy $c_i \equiv 1 \pmod 3$, and $R$ satisfy $c_i \equiv 2\pmod 3$.  Then the stick number $s(L)$ of $L$ satisfies

\[
s(L) \le \frac{2}{3}c(L) + 2n + 3 - \frac{2}{3}Q - \frac{1}{3}R.
\]

\end{thm}

\noindent
\begin{proof}  
The portion of the graph outside of the $3$-balls consists of $2n+3$ sticks. 
In each $3$-ball, a supercoiled tangle uses four sticks from the outside $2n+3$ sticks.  
Hence each of these tangles uses four less sticks than the stick numbers
stipulated by Lemma \ref{lemma:IntegralTangle}.
The construction yields the upper bound

\begin{align}
s(L) &\le 2n + 3 + \sum_{c_i \equiv_3 0}\frac{2}{3}c_i + \sum_{c_j \equiv_3 1  }\left(\frac{2}{3}c_j-\frac{2}{3}\right) + \sum_{c_k \equiv_3 2 }\left( \frac{2}{3}c_k-\frac{1}{3}\right) \nonumber\\
&= \sum_{i=1}^n\frac{2}{3}c_i + 2n + 3 - \frac{2}{3}Q - \frac{1}{3}R\nonumber\\
&=\frac{2}{3}c(L) + 2n + 3 - \frac{2}{3}Q - \frac{1}{3}R.\nonumber
\end{align} \end{proof}

We make a few observations to put this result in context.  This improves on the previous upper bound 
as long as the upper bound of Theorem~\ref{thm:2BridgeUpperBound} is less than $c(L) + 2$.  Solving
that inequality for crossing number, we see that Theorem~\ref{thm:2BridgeUpperBound} is an improvement for

\[
6n + 1 - 2Q - R < c(L).
\]
Intuitively, if there are at least six times as many crossings as tangles, the supercoiling technique is preferable.  

In particular, if each twist has more than seven crossings supercoiling provides a better construction.  It is
interesting to note that similar bounds on crossings in twists appear when considering hyperbolic structures 
on knots. Purcell proves that any knot or
link with at least 6 crossings in each twist region is hyperbolic if its augmented link is. Moreover,
she obtains twist number bounds on hyperbolic volume when each twist region has at least seven 
crossings (see \cite{pu}).  One wonders if there is a connection between hyperbolic volumes, supercoiling, 
and stick number.

Secondly, let us interpret Theorem~\ref{thm:2BridgeUpperBound} in the case where the number of tangles is 
fixed but the crossing number increases without bound.  For fixed $n$ and large $c$, Theorem~\ref{thm:2BridgeUpperBound} says that the stick number is asymptotically at worst two-thirds
the crossing number.  Previous upper bounds predicted that stick number is at worst asymptotic to the crossing number.
An obvious question is whether or not the two-thirds bound is optimal.

\section{Minimal Crossing Projections} \label{crossings}

Both McCabe and Huh et al. point out that their constructions produce polygonal links which admit minimal crossing projections \cite{hno, mc}. 
 In contrast, in this section we show that minimal stick representatives of 2-bridge links do not admit minimal crossing projections in general.

Minimal stick representatives of $(2,n)$ torus links do not admit minimal crossing projections for $n \ge 15$ \cite[Corollary 2]{jmt}.  
The main observation used is that straightening the boundary of a bigon in a diagram into sticks requires introducing a vertex.  We formalize this in a lemma.

\begin{lemma}\label{lemma:BigonsAndSticks}
Let $D$ be a projection of the link $L$ with $b$ bigons.  Any polygonal representative of $L$ admitting a projection planar isotopic to $D$ has at least $b$ sticks.
\end{lemma}

\noindent
\textit{Proof.}  Let \textbf{L} be a polygonal representative of $L$ which admits a projection $D_{\textbf{L}}$ isotopic to $D$.  
Since every bigon of $D$ has two boundary arcs, when straightening the arcs in the isotopy to $D_{\textbf{L}}$ at least one vertex must be introduced.  
Thus \textbf{L} must have at least as many vertices, and hence edges, as bigons in the projection $D$. $\Box$

The next lemma analyzes minimal crossing projections of 2-bridge links to obtain a lower bound on the number of bigons in any minimal crossing projection.

\begin{lemma}\label{lemma:RatlLinkBigons}
Let $L$ be a rational link given by the integers $c_1, c_2, \dots,c_n$, with each $c_i \ge 2$.  
Then any minimal crossing diagram of $L$ has at least $c(L) -2n + 2$ bigons.
\end{lemma}

\noindent
\textit{Proof}.  The Tait Flyping conjecture, proven by Menasco and Thistlethwaite in 1987, 
states that any alternating diagram of $L$ can be obtained from the standard one by flypes \cite[Theorem 1]{mt}.  
Since $c_i \ge 2$, all bigons in the standard diagram occur in the integral tangles, and there are $c_i - 1$ bigons in each. (If either $c_1$ or $c_n$ is one, 
it forms a bigon with the adjacent tangle.)  Flypes can reduce the number of bigons by distributing the crossings of an integral tangle throughout its flype orbit.  
To obtain the lower bound, then, we analyze the flype orbits in rational links.

\begin{figure}[h]
 \centering
 \includegraphics[width=5in]{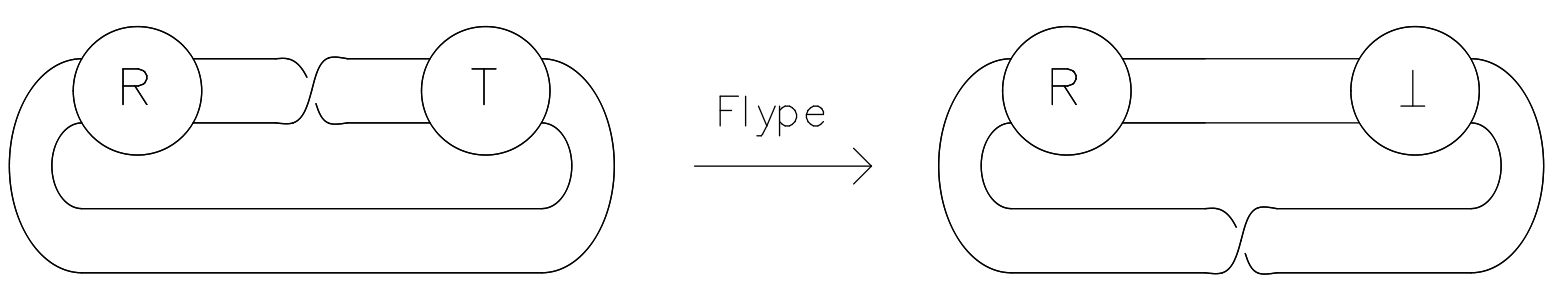}
 % Figure100.pdf: 3370x2383 pixel, 72dpi, 118.89x84.07 cm, bb=0 0 3370 2383
 \caption{A Flype} \label{100}
\end{figure}

To specify a non-trivial flype for a crossing in a knot diagram, the complement of the crossing must be decomposed into two non-trivial 4-tangles as in Figure~\ref{100}.  
Since all crossings in an integral tangle admit the same flypes, 
we will speak of tangles as admitting flypes.  We now classify the non-trivial flypes in the standard diagram of a 2-bridge link (see \cite{ca} for more details regarding flype orbits).

\begin{figure}[h]
 \centering
 \includegraphics[width=4in]{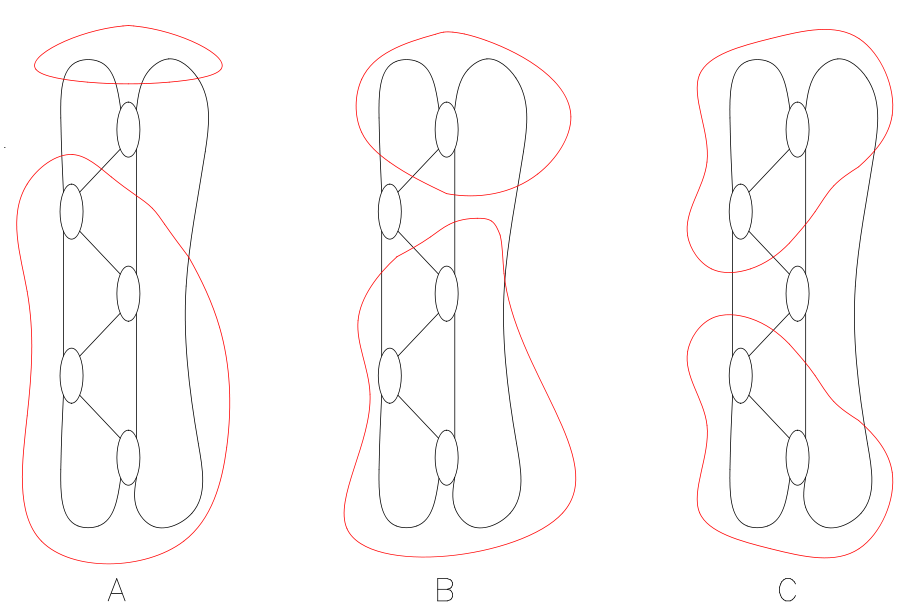}
\caption{Flype Orbits} \label{101}
 % Figure101.pdf: 3370x2383 pixel, 72dpi, 118.89x84.07 cm, bb=0 0 3370 2383
\end{figure}

The first and last tangles do not admit non-trivial flypes because in the only decomposition of the complement into 4-tangles one of them is trivial (see Figure~\ref{101}a).  
Thus these tangles contribute $c_1 -1$ and $c_n -1$ bigons to any alternating diagram of $L$.  Each intermediate tangle admits a unique non-trivial flype as in Figure~\ref{101}b and \ref{101}c.  
Thus for $1 < i < n$ the $c_i$ crossings can be distributed to two tangles with $c_i '$ and $c_i ''$ crossings respectively. 
 The two tangles then contribute $c_i' -1$ and $c_i '' -1$ bigons respectively, for a total of $c_i ' - 1 + c_i'' - 1 = c_i -2$ bigons.  
(See Figure~\ref{102} for a diagram with distributed crossings.)  
This is the most that crossings can be distributed, since non-trivial flypes are unique, giving a lower bound for the number of bigons.  
Summing we see that the fewest number of bigons in any alternating (hence minimal crossing) diagram of $L$ is

\[
c_1 + c_n -2 + \sum_{1<i<n}\left(c_i - 2\right) = \sum_{i=1}^n c_i - 2 -2(n-2) = c(L) -2n +2.\ \Box
\]

\begin{figure}[h]
 \centering
 \includegraphics[width=5in]{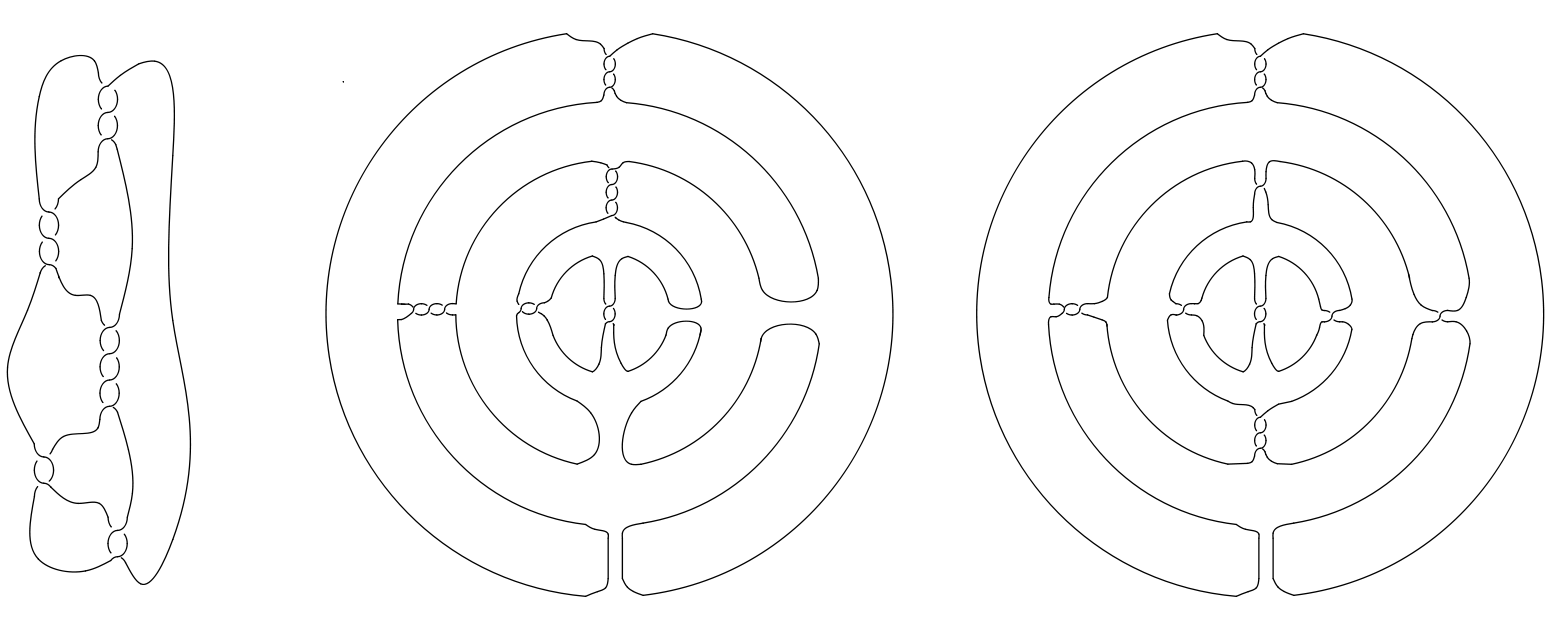}
 % Figure101.pdf: 3370x2383 pixel, 72dpi, 118.89x84.07 cm, bb=0 0 3370 2383
\caption{Different Projections of a 2-Bridge Knot Suggesting a Flype Orbit} \label{102}
\end{figure}

Combining Lemmas \ref{lemma:BigonsAndSticks} and \ref{lemma:RatlLinkBigons} yields the following result.

\begin{thm} \label{thm:minimalcrossing}
Let $L$ be a rational link given by the integers $c_1, c_2, \dots,c_n$, with each $c_i \ge 2$.  Further, suppose $P$ of the integers satisfy 
$c_i \equiv 0 \pmod 3$, $Q$ satisfy $c_i \equiv 1 \pmod 3$, and $R$ satisfy $c_i \equiv 2\pmod 3$.  Minimal stick representatives of $L$ do not admit minimal crossing projections whenever

\[
12n + 3 -2Q-R < c(L).
\]
\end{thm}

\noindent
\begin{proof} By Lemma \ref{lemma:RatlLinkBigons}, there are at least $c(L) - 2n + 2$ bigons in any minimal crossing projection of a rational link $L$.  
By Lemma \ref{lemma:BigonsAndSticks}, if the relationship $s(L) < c(L) -2n+2$ is satisfied then minimal stick representatives of $L$ do not admit alternating projections.  
Setting the upper bound of Theorem \ref{thm:2BridgeUpperBound} less than $c(L) -2n+2$ gives a range for which minimal stick representatives do not admit minimal crossing projections.  We have

\[
s(L) \le \frac{2}{3}c(L) + 2n + 3 - \frac{2}{3}Q - \frac{1}{3}R < c(L) - 2n +2.
\]

\noindent
Solving the right hand inequality for $c(L)$ completes the proof. \end{proof}

Loosely speaking, then, minimal stick representatives of rational links do not admit minimal crossing projections as long as there are at least twelve times as many crossings as tangles.  
In the extreme case when every $c_i$ is a multiple of three, i.e. when $P = n$, the bound reduces to $12n + 3 < c(L)$.  The opposite extreme, when $Q=n$, yields the bound $10n+3 < c(L)$.
 
\begin{figure}[h]
\begin{center}
 \includegraphics[width=4in]{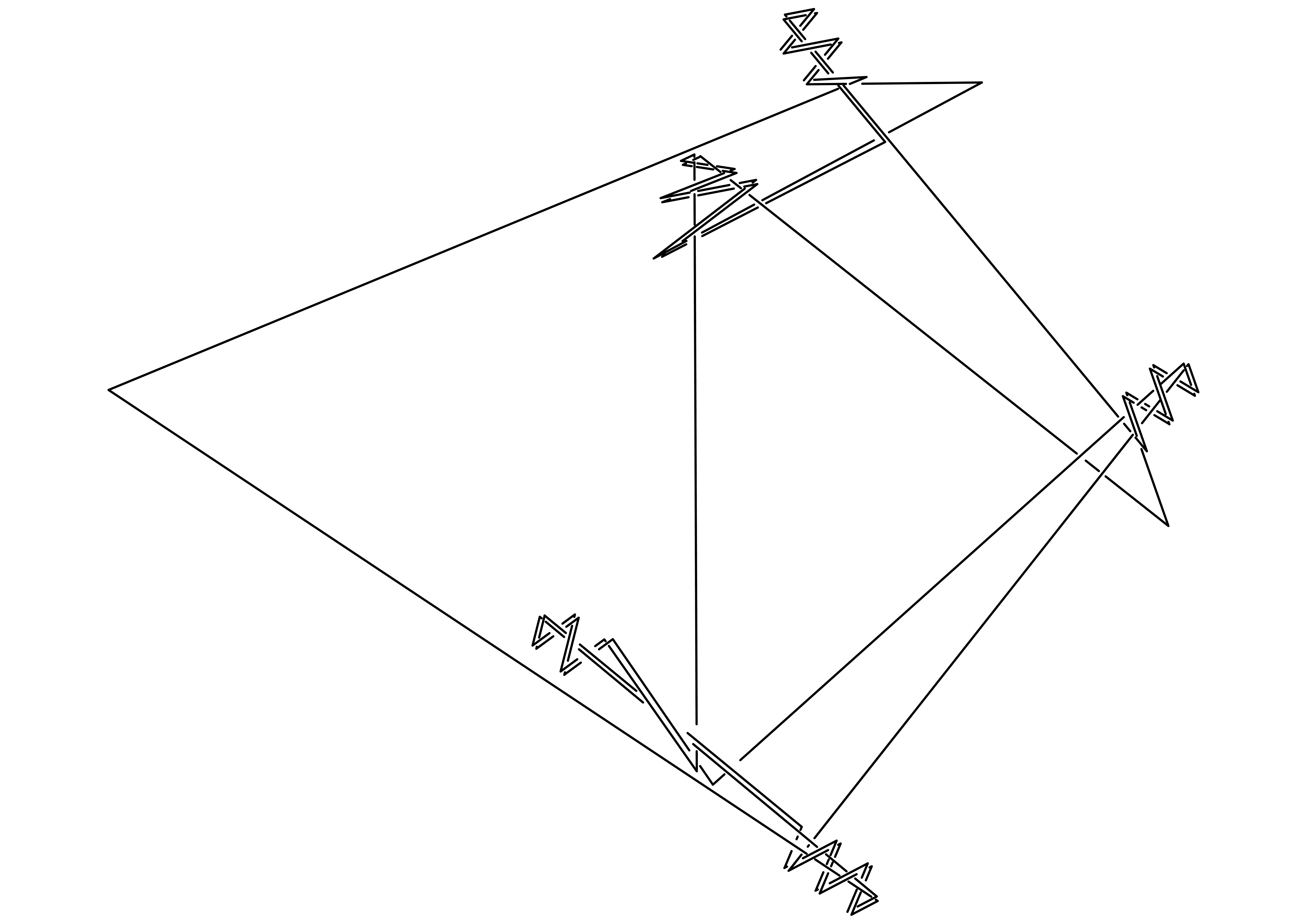}
 % ConstructionExample1.pdf: 3370x2383 pixel, 72dpi, 118.89x84.07 cm, bb=0 0 3370 2383
\end{center}
\caption{A polygonal 2-Bridge link with no minimal crossing projection}
\label{fig:5Tangles}
\end{figure}

In the example pictured in Figure~\ref{fig:5Tangles}, each of the five tangles has $16$ crossings.
Hence $n=Q =5$ and the crossing number is $c(L) = 80$. The stick number is $s(L) \le 63$ and the minimum number of bigons in any minimal crossing projection is $72$.
Thus, no minimal stick representation of this link can yield a minimal crossing projection.

\noindent
\textbf{Acknowledgements}

This research was supported in part by NSF-REU Grants DMS-0453605 and DMS-1156608, as well as California State University, San Bernardino.

\end{document}